\documentclass[12pt]{article}
\usepackage[utf8]{inputenc}
\usepackage{amsmath,amssymb,amsfonts,amsthm,fancyhdr}
\usepackage[numbers,sort&compress]{natbib}
\usepackage{url}
\usepackage{ifthen}
\usepackage{color}
\usepackage{times}
\usepackage{ifpdf}
\usepackage{enumitem}
\usepackage{fullpage}
\usepackage{graphicx}

\newtheorem{theorem}{Theorem}
\newtheorem{lemma}[theorem]{Lemma}
\newtheorem{corollary}[theorem]{Corollary}
\newtheorem{conjecture}{Conjecture}

\newcommand{\edim}{{\rm edim}}
\newcommand{\mdim}{{\rm mdim}}

\title{On Metric Dimensions of Hypercubes\footnotetext{\small\tt aleksander.kelenc@um.si, aodenteo@gmail.com, skrekovski@gmail.com,\\ \hspace*{0.62cm}and ismael.gonzalez@uca.es}}

\author{Aleksander Kelenc$^{1,2}$, Aoden Teo Masa Toshi$^3$, Riste \v Skrekovski$^{4,5}$,\\ and Ismael G. Yero$^6$\\[0.3cm]
\small $^1$ \it University of Maribor, FERI, 2000 Maribor, Slovenia \\[0.1cm]
\small $^2$ \it Institute of Mathematics, Physics and Mechanics, 1000 Ljubljana, Slovenia \\[0.1cm]
\small $^3$ \it Independent researcher, Singapore\\[0.1cm]
\small $^4$ \it University of Ljubljana, FMF, 1000 Ljubljana, Slovenia \\[0.1cm]
\small $^5$ \it Faculty of Information Studies, 8000 Novo Mesto, Slovenia \\[0.1cm]
\small $^6$ \it Universidad de C\'adiz, Departamento de Matem\'aticas, Algeciras, Spain\\
}

\date{\today}

\begin{document}

\maketitle
\begin{abstract}
The metric (resp. edge metric or mixed metric) dimension of a graph $G$, is the cardinality of the smallest ordered set of vertices that uniquely recognizes all the pairs of distinct vertices (resp. edges, or vertices and edges) of $G$ by using a vector of distances to this set. In this note we show two unexpected results on hypercube graphs. First, we show that the metric and edge metric dimension of $Q_d$ differ by only one for every integer $d$. In particular, if $d$ is odd, then the metric and edge metric dimensions of $Q_d$ are equal. Second, we prove that the metric and mixed metric dimensions of the hypercube $Q_d$ are equal for every $d \ge 3$. We conclude the paper by conjecturing that all these three types of metric dimensions of $Q_d$ are equal when $d$ is large enough.
\end{abstract}



\section{Introduction}

The metric dimension of connected graphs was introduced about 50 years ago in \cite{Harary1976,Slater1975}, in connection with modeling navigation systems in networks, although this notion was already known by then for general metric spaces from \cite{Blumenthal1953}. Given a connected graph $G$ and two vertices $u,v\in V(G)$, the \emph{distance} $d_G(u,v)$ between these two vertices is the length of a shortest path connecting $v$ and $u$. The vertices $u,v$ are \emph{distinguished} (\emph{determined} or \emph{resolved}) by a vertex $x\in V(G)$ if $d_G(u,x)\ne d_G(v,x)$. A given set of vertices $S$ is a \emph{metric generator} for the graph $G$, if every two vertices of $G$ are distinguished by a vertex of $S$. The cardinality of the smallest possible metric generator for $G$ is the \emph{metric dimension} of $G$, which is denoted by $\dim(G)$. The terminology of metric generators was introduced in \cite{Khuller1996}, since the previous two works referred to such sets as \emph{resolving sets} and \emph{locating sets}, respectively. A metric generator for $G$ of cardinality $\dim(G)$ is called a \emph{metric basis}. Although the classical metric dimension is an old topic in graph theory, there are still several open problems that remain unsolved. Recent investigations on this concern are \cite{Geneson2020,Gutin2020,Jiang2019,Laird2020}.

In order to uniquely identify the edges of a graph, by using vertices, the edge metric dimension of connected graphs was introduced in \cite{Kelenc2018} as follows. Let $G$ be a connected graph and let $uv$ be an edge of $G$ such that $u,v \in V(G)$. The \emph{distance} between a vertex $x\in V(G)$ and the edge $uv$ is defined as, $d_G(uv, x) = \min \{d_G(u,x), d_G(v,x)\}.$
It is said that two distinct edges $e_1, e_2 \in E(G)$ are \emph{distinguished} or \emph{determined} by a vertex $v\in V(G)$ if $d_G(e_1, v) \ne d_G(e_2, v).$ A set $S\subset V(G)$ is called an \emph{edge metric generator} for $G$ if and only if for every pair of edges $e_1, e_2 \in E(G)$, there exists an element of $S$ which distinguishes the edges. The cardinality of a smallest possible edge metric generator of a graph is known as the \emph{edge metric dimension}, and is denoted by $\edim(G)$. After the seminal paper \cite{Kelenc2018}, a significant number of researches on such parameter have appeared. Among them, some of the most recent ones are \cite{Geneson2020,Klavzar2020,Knor2021-a,Knor2021-b,Peterin2020}. It is natural to consider comparing the metric and edge metric dimensions of graphs. However, as first proved in \cite{Kelenc2018}, and continued in \cite{Knor2021-a,Knor2021-b}, both parameters are not in general comparable existing connected graphs $G$ for which $\edim(G)<\dim(G)$, $\edim(G)=\dim(G)$ or $\edim(G)>\dim(G)$.

In order to combine the unique identification of vertices and of edges, in only one scheme, the mixed metric dimension of graphs was introduced in \cite{Kelenc2017}. For a connected graph $G$, a vertex $w\in V(G)$ and an edge $uv\in E(G)$ are \emph{distinguished} or \emph{determined} by a vertex $x\in V(G)$ if $d_G(w, x) \ne d_G(uv, x)$. A set $S\subset V(G)$ is called an \emph{mixed metric generator} for $G$ if and only if for every pair of elements of the graphs (vertices or edges) $e,f \in E(G)\cup V(G)$, there exists a vertex of $S$ which distinguishes them. The cardinality of a smallest possible mixed metric generator of $G$ is known as the \emph{mixed metric dimension} of $G$, and is denoted by $\mdim(G)$. Some recent studies on mixed metric dimension of graphs are \cite{Sedlar2020-a,Sedlar2020-b}. Clearly, every mixed metric generator must be a metric generator as well as an edge metric generator, and so, $\mdim(G)\ge \max\{\dim(G),\edim(G)\}$, for any connected graph $G$. In consequence (and due to the not comparability between $\dim(G)$ and $\edim(G)$), several situations relating these three parameters can be found. That is, graphs $G$ with $\mdim(G)>>\max\{\dim(G),\edim(G)\}$, $\mdim(G)=\dim(G)>>\edim(G)$, $\mdim(G)=\edim(G)>>\dim(G)$, or $\mdim(G)=\dim(G)=\edim(G)$.

The metric dimension of hypercube graphs has attracted the attention of several researchers from long ago. For instance, the work of Lindstr\"om \cite{Lindstrom1965} is probably one of the oldest ones, and for some recent ones we suggest the works \cite{Hertz2020,Nikolic2017,Zhang2020}. Surprisingly, for other related invariants, there has been comparatively little research on hypercube graphs, although one can find some interesting recent results on this topic, such as those that appear in \cite{Hakanen2020,Hertz2020}. It is our goal to present some results on the closed connections that exist among the metric, edge metric and mixed metric dimensions of hypercube graphs.

The $d$-dimensional hypercube, denoted by $Q_d$, with $d\in \mathbb{N}$, is a graph whose vertices are represented by $d$-dimensional binary vectors, \emph{i.e.}, $u=(u_1,\dots,u_2)\in V(Q_d)$ where $u_i\in \{0,1\}$ for every $i\in\{1,\dots,d\}$. Two vertices are adjacent in $Q_d$ if their vectors differ in exactly one coordinate. Hypercubes can be also seen as the $d$ times Cartesian product of the graph $P_2$, that is, $Q_d\cong P_2\Box P_2\Box \cdots \Box P_2$, or recursively, $Q_d\cong Q_{d-1}\Box P_2$. The distance between two vertices in $Q_d$ represents the total number of coordinates in which their vectors differ. The hypercube $Q_d$ is bipartite, and has $2^d$ vertices and $d\cdot 2^{d-1}$ edges. We remark that for instance, $Q_2$ is the cycle $C_4$, and that $Q_4$ can be also seen as the torus graphs $C_4\Box C_4$.

\section{Results}

Our first contribution is to relate the metric generators with the edge metric generators of bipartite graphs.

\begin{lemma}
\label{lem:edim-less-dim}
Let $G$ be a connected bipartite graph. Then, every metric generator for $G$ is also an edge metric generator.
\end{lemma}

\begin{proof}
Let $S$ be an arbitrary metric generator for $G$. We will show that $S$ is an edge metric generator as well.

Let $e_1 = x_1 y_1$ and $e_2=x_2y_2$ be two arbitrary distinct edges of $G$. It is easy to choose $u\in \{x_1,y_1\}$ and $v\in \{x_2,y_2\}$ such that $u$ and $v$ are distinct and on even distance. In case $e_1$ and  $e_2$ share a vertex, say $y_1=y_2$, then let $u=x_1$ and $v=x_2$; as these vertices are on distance 2, the requirements are satisfied. And, in case $e_1$ and $e_2$ have no common vertex, then simply let $u=x_1$ and observe that precisely one of $x_2$ and $y_2$ is on even distance from $u$ as $G$ is a bipartite graph and $y_1$ and $y_2$ are adjacent. Without lost of generality assume $d(y_2,u)$ is even, and finally set $v=y_2$.

Now, as  $u$ and $v$ are distinct, there must be a vertex $s\in S$ that distinguishes them, \emph{i.e.} $d(s,u)\ne d(s,v)$. We may assume
that $d(s,u)+1\le d(s,v)$. Since $u$ and $v$ are on even distance, it follows that distances $d(s,u)$ and $d(s,v)$ are of same parity, otherwise we encounter a closed walk of odd length in $G$, which is not possible in a bipartite graph. This implies $d(s,u)+2\le d(s,v)$, and now we easily derive
              $$d(e_1,s)\le d(u,s) < d(v,s)-1\le d(e_2,s).$$
In particular,  $d(e_1,s)< d(e_2,s)$ implies that $e_1,e_2$ are distinguished by $s\in S$. Since the choice of these two edges was arbitrary,
we conclude that $S$ is also an edge metric generator.
\end{proof}

It is then natural to think in the opposite direction with regard to the result above. In particular, we ask if an edge metric generator for a bipartite graph is also a metric generator. In contrast with the result above, achieving this seems to be a challenging task. However, we have at least managed to show a weaker result for an infinite family of bipartite graphs, namely the hypercubes $Q_d$. That is, when $d$ is odd, every edge metric generator for $Q_d$ is indeed a metric generator, and when $d$ is even, every edge metric generator is ``almost” a metric generator.

From now on we denote by $\alpha_i$ the vector of dimension $d$ whose $i^{\rm th}$-coordinate is 1, and the remaining coordinates are 0. Also, by ``$\oplus$'' we represent the standard (binary) XOR operation. Notice that, for any vertex $u\in V(Q_d)$, $u\oplus \alpha_i$ means switching the $i^{\rm th}$-coordinate of $u$ from 0 to 1, or vice versa.

\begin{lemma}
\label{lem:edge-dim-equal-odd}
Let $S$ be  an edge metric generator of  $Q_d$. If $u$ and $v$ are two vertices not distinguished by $S$, then
they must be antipodal in $Q_d$ and $d$ is even. In particular,  $S$ is  a metric generator of $Q_d$ for $d$ being odd.
\end{lemma}

\begin{proof}
Suppose that $u=(u_1,u_2,\ldots,u_d)$ and $v=(v_1,v_2,\ldots,v_d)$ are not antipodal. Then, $u_i=v_i$ for some $i$.
Let $Q^0_{d-1}$ and $Q^1_{d-1}$ be the half-cubes regarding the dimension $i$. Notice that $u$ and $v$ belongs to a same
half-cube, say  $Q^0_{d-1}$.  Let $e_u$ and $e_v$ be the edges corresponding to the component $i$ (in $Q_d$) incident with $u$ and $v$, respectively.
In other words, as $u\oplus \alpha_i$ and $v\oplus \alpha_i$ are the neighbours of $u$ and $v$ in $Q_{d-1}^1$, we have
$e_u = (u,u\oplus \alpha_i)$ and $e_v = (v,v\oplus \alpha_i)$. We claim that the edges
 $e_u$ and $e_v$ are not distinguished by $S$. To see this, observe that if $s\in S$ belongs to $Q^0_{d-1}$, then
    $$d(s,e_u) = d(s,u) = d(s,v)= d(s,e_v).$$
Also, if $s\in S$ belongs to $Q^1_{d-1}$, then
    $$d(s,e_u) = d(s,u\oplus \alpha_i) = d(s,v\oplus \alpha_i)= d(s,e_v).$$
 We hence derive that edges  $e_u$ and $e_v$ are not distinguished by $S$, which is a contradiction.

Based on the above arguments, we conclude that $u$ and $v$ are antipodal, \emph{i.e.} $d(u,v)=d$. Hence, every vertex $x$ of $S$ satisfies
$d(u,x) + d(x,v)=d$. As every vertex $s\in S$ must be equally distanced from $u$ and $v$, we conclude that $d(u,s) =d(s,v)=d/2$, and consequently, $d$ must be even. This establishes the main claim.

Finally, observe that if $d$ is odd, then no vertex is equally distanced from two antipodal vertices of $Q_d$, and therefore, $S$ is a metric generator of $Q_d$.
\end{proof}

Next lemma will ensure that enlarging an edge metric generator of $Q_d$ with one chosen element, we get a metric generator of $Q_d$.

\begin{lemma}
\label{lem:edge-gen-dim-plus-one}
Let $S$ be an edge metric generator of $Q_d$ and let $s$ be an arbitrary element of $S$. Then,
$S\cup \{s\oplus \alpha_1\}$ is a metric generator of $Q_d$.
\end{lemma}

\begin{proof}
If $S$ is a metric generator of $Q_d$, then $S\cup \{s\oplus \alpha_1\}$ is so too, and we are done. Thus, we assume that $S$ is not a
metric generator of $Q_d$. Then, by Lemma \ref{lem:edge-dim-equal-odd}, $d$ is even and there must exist antipodal vertices $u$ and $v$
such that $d(u,x) = d(v,x)=d/2$ for every $x\in S$. This will not be a case for $s\oplus \alpha_1$, as $|d(u,s\oplus \alpha_1) - d(v,s\oplus \alpha_1)|=2$.
Therefore, we conclude that $S\cup \{s\oplus \alpha_1\}$ is a metric generator of $Q_d$.
 \end{proof}

Since $Q_d$ is a bipartite graph, the two previous lemmas give us the following consequence.

\begin{theorem}
\label{th:dim-edim}
For every $d\ge 1$ it holds
            $$\edim(Q_d) \le  \dim(Q_d) \le \edim(Q_d) + 1, $$
with the second inequality being tight only if $d$ is even.
\end{theorem}

\begin{proof}
The lower bound holds by Lemma \ref{lem:edim-less-dim}. The upper bound and its possible tightness (for more than one case) follows by
Lemmas \ref{lem:edge-dim-equal-odd} and \ref{lem:edge-gen-dim-plus-one}.
 \end{proof}

Notice that the upper bound $\dim(Q_d) \le \edim(Q_d) + 1$ is indeed tight for the case $Q_4$, since $4=\dim(Q_4)=\edim(Q_4) + 1$, as proved in \cite{Kelenc2018}.\\

We now turn our attention to relating the metric dimension with the mixed metric dimension of hypercubes. To this end, we will need the following two results. We must remark that the first of next two lemmas already appeared in \cite{Nikolic2017}. We include its proof here by completeness of our paper.

\begin{lemma}
\label{lemma:antipodal_basis}
If $S$ is a resolving set (metric basis) of $Q_d$ and $s \in S$, then $(S\setminus \{s\}) \cup \{s'\}$ is also a resolving set (metric basis) of $Q_d$, where $s'\in V(Q_d)$ is the antipodal vertex of $s$.
\end{lemma}

\begin{proof}
If $s\in S$ distinguishes some pair of vertices $x$ and $y$ of $Q_d$, then $s'$ distinguishes such pair as well, since $d(x,s') =  d - d(x,s)$ and $d(y,s') =  d - d(y,s)$. This also means that no metric basis of $Q_d$ contains two antipodal vertices. Thus, if $S$ is a metric generator (or a metric basis) of $Q_d$, then $S\setminus \{s\}\cup \{s'\}$ is so too.
\end{proof}

\begin{lemma}\label{lemma:two_columns}
If $S$ is a metric generator of $Q_d$, then there is at most one index $i \in \{ 1, \ldots, d \}$ such that all the vertices from $S$ have the same value at the $i^{\rm th}$ coordinate.
\end{lemma}

\begin{proof}
Suppose that there exist two different indices $i$ and $j$ such that all vertices from $S$ have the same value at the $i^{\rm th}$ and $j^{\rm th}$ coordinates. In the proof, we proceed with the case when there are zeroes at such coordinates. The other cases can be shown by using similar arguments. Now, let $x\in V(Q_d)$ be a vertex having zeroes at all coordinates, except at the $i^{\rm th}$, and let $y$ be a vertex having zeroes at all positions except at the $j^{\rm th}$. Then, $d(x,s)=d(y,s)$ for any vertex $s \in S$, a contradiction.
\end{proof}

The mixed metric dimension of hypercubes $Q_1$ and $Q_2$ are $2$ and $3$, respectively. This can be derived from results for paths and cycles from \cite{Kelenc2017}. This gives us that $\dim(Q_d)<\mdim(Q_d)$, for $d \in\{1,2\}$. For all higher dimensions the mixed metric dimension is equal to the metric dimension as we next show.

\begin{theorem}
\label{th:dim-mdim}
For every $d\ge 3$ it holds,
            $$\dim(Q_d) = \mdim(Q_d).$$
\end{theorem}

\begin{proof}
First, $\{(1,1,1),(0,1,0), (0,0,1)\}$ and $\{(1,1,1,1),(0,1,0,0),(0,0,1,0),(0,0,0,1)\}$ are mixed metric bases for $Q_3$ and $Q_4$, respectively. Thus, the equality follows for these cases since $\dim(Q_3)=3$ and $\dim(Q_4)=4$. It remains to check the equality for $d\geq 5$.

Let $S$ be a metric basis for $Q_d$ with $d\geq 5$.
By Lemma \ref{lem:edim-less-dim}, $S$ is an edge metric generator of $Q_d$. In this sense, we only need to distinguish in $Q_d$, those pairs of elements, one of them being a vertex and the other one, an edge. For this, let $u$ be an arbitrary vertex and let $e=xy$ be an arbitrary edge of $Q_d$.

Suppose first that $u$ is not a vertex of $e$. As $d(u,x)$ and $d(u,y)$ are of different parity, we may assume that
$u$ and $x$ are on even distance. Now, let $s_i$ be a vertex from $S$ that distinguishes $u$ and $x$. Similarly, as in Lemma \ref{lem:edim-less-dim}, notice that $d(s_i,u)$ and $d(s_i,x)$ are of the same parity, and as they are different, we have that $|d(s_i,u) - d(s_i,x)|\ge 2$.
So, if $d(s_i,u)<d(s_i,x)$, then we derive
               $$ d(s_i,u)<d(s_i,u)+1\le d(s_i,x)-1\le d(s_i,e),$$
and if $d(s_i,x)<d(s_i,u)$, then we have
              $$ d(s_i,e) \le  d(s_i,x) <  d(s_i,u).$$
Thus, in both cases $e$ and $u$ are distinguished by a vertex from $S$.

So all the pairs of elements (vertices and edges) considered in the upper part are distinguished by an arbitrary metric basis. To conclude the proof, we need to construct a metric basis of cardinality $\lvert S \rvert $ that will also distinguish incident vertices and edges.

Suppose now that $u$ is an endpoint of $e$, say $u=x$. To distinguish $u$ and $e$ there needs to be a vertex $s \in S$ which is from the half-cube $Q_{d-1}$ that contains vertex $y$ and does not contain vertex $x$. To distinguish all such pairs there must be at least one vertex from the mixed metric generator in every half-cube $Q_{d-1}$. For any index $i \in \{1, \ldots, d\}$, there exists a vertex from a mixed metric generator having $0$ on the $i^{\rm th}$ coordinate, and a vertex from a mixed metric generator having $1$ on the $i^{\rm th}$ coordinate. In other words, a mixed metric basis does not have a column of zeroes or a column of ones at an arbitrary index $i$ (if we arrange all vectors of the mixed metric basis as a matrix with such vectors as the rows of such matrix).

We have started from an arbitrary metric basis $S$. Since $Q_d$ is a vertex transitive graph, we may assume that the vertex $s_1=(0,0,\ldots,0)$ (all coordinates equal to 0) is in $S$. If $S$ does not contain a column of zeroes, then $S$ is also a mixed metric basis. Otherwise, by Lemma \ref{lemma:two_columns}, there exists only one such column, say at index $i_0$. By Lemma \ref{lemma:antipodal_basis}, we know that we can replace any of the vertices from the set $S$ with its antipodal vertex and the incurred set $S'=S \setminus \{s\} \cup \{s'\}$ is a metric basis too, since the column at index $i_0$ (all zeroes) ensures that no two vertices in $S$ are antipodal to each other. Moreover, in view of Lemma \ref{lem:edim-less-dim}, $S$ is an edge metric generator as well.

There exist at least four different vertices $s_1=(0,0, \ldots, 0)$, $s_2$, $s_3$ and $s_4$ in the set $S$, since $\dim(Q_d)\geq 4$, for $d\geq 5$. We construct four sets $S_i'$ in the following way:
$$S_1'=\left( S \setminus \{s_1\} \right) \cup \{s_1'\},\;\;\;\; S_2'=\left( S \setminus \{s_2,s_3\} \right) \cup \{s_2', s_3'\}, $$
$$S_3'=\left( S \setminus \{s_2\} \right) \cup \{s_2'\}, \;\;\;\; S_4'=\left( S \setminus \{s_1,s_3\} \right) \cup \{s_1', s_3'\}, $$
and consider the next situations:

\noindent
{\bf (I):} If $S_1'$ is not a mixed metric generator, then there is a column of ones in $S_1'$ at some index $i_1$.

\noindent
{\bf (II):} If $S_2'$ is not a mixed metric generator, then there is a column of zeroes in $S_2'$ at some index $i_2$.

\noindent
{\bf (III):} If $S_3'$ is not a mixed metric generator, then there is a column of zeroes in $S_3'$ at some index $i_3$.

\noindent
{\bf (IV):} If $S_4'$ is not a mixed metric generator, then there is a column of ones in $S_4'$ at some index $i_4$.

Observe that all these indices $i_0$, $i_1$, $i_2$, $i_3$, and $i_4$ are different. If none of the four sets $S_i'$ defined above is a mixed metric generator, then the set initial $S$ looks as follows.

\begin{center}
\begin{tabular}{ c c | c c c c | c c c }
	   & $i_0$ & $i_1$ & $i_2$ & $i_3$ & $i_4$ &  & \ldots &  \\
	   \hline
$s_1:$ & 0 & 0 & 0 & 0 & 0 &  & \ldots &  \\
$s_2:$ & 0 & 1 & 1 & 1 & 1 &  & \ldots &  \\
$s_3:$ & 0 & 1 & 1 & 0 & 0 &  & \ldots &  \\
\hline
$s_4:$ & 0 & 1 & 0 & 0 & 1 &  & \ldots &  \\
\vdots & \vdots & \vdots & \vdots & \vdots & \vdots &   &  &   \\
$s_{|S|}:$ & 0 & 1 & 0 & 0 & 1 &  & \ldots &  \\
\end{tabular}
\end{center}

We now take a look at the columns $i_1$, $i_2$, $i_3$ and $i_4$. Let $v_1$ be a vertex having zeroes at all positions except at $i_1$ and $i_3$ and let $v_2$ be a vertex having zeroes at all positions except at $i_2$ and $i_4$. Then, $d(v_1,s)=d(v_2,s)$, for any vertex $s \in S$, a contradiction. Therefore, at least one of the sets $S_i'$ has to be a mixed metric generator, and therefore, the equality $\mdim(Q_d)=\dim(Q_d)$ follows since any mixed metric basis is also a metric basis.
\end{proof}

Theorems \ref{th:dim-edim} and \ref{th:dim-mdim} give us the following consequences, by using also the asymptotical result for the metric dimension of hypercubes known from \cite{Cantor1966}.

\begin{corollary}
For every $d\ge 3 $ it holds
$$\dim(Q_d) -1 \le  \edim(Q_d) \le \dim(Q_d) = \mdim(Q_d).$$
\end{corollary}

\begin{corollary}
For every $d\ge 2$ it holds
$$\mdim(Q_d) \sim \edim(Q_d) \sim \dim(Q_d) \sim \frac{2d}{\log_2 d}.$$
\end{corollary}

We conclude this little paper with the following conjecture.

\begin{conjecture}
If $d$ is large enough, then $$\edim(Q_d)=\dim(Q_d).$$
 \end{conjecture}

As the above conjecture does not hold for $d=4$, $d$ must be at least 5.

\section*{Acknowledgements}

A. Kelenc has been partially supported by the Slovenian Research Agency by the projects No. J1-1693 and J1-9109. R. \v Skrekovski acknowledges the Slovenian research agency ARRS, program no.\ P1--0383 and project no. J1-1692. Ismael G. Yero has been partially supported by the Spanish Ministry of Science and Innovation through the grant PID2019-105824GB-I00.

\end{document}